\documentclass[reqno]{amsart}
\usepackage{amssymb,amsthm} 
\usepackage[english,activeacute]{babel}
 
 \usepackage{url,csquotes}

 \usepackage{lineno,hyperref}
 
\hypersetup{bookmarks=true,
	unicode=true,
	colorlinks=true,
	citecolor=black,
	linkcolor=black,
	urlcolor=black,
	plainpages=false,
	pdfpagelabels=true}
\usepackage{mathptmx,pxfonts,amssymb,amsfonts,amsmath,graphicx,shadow,color}

\usepackage{mathtools}
\usepackage{ bbold } 

\allowdisplaybreaks 

\newtheorem{theorem}{Theorem}[section]

\theoremstyle{definition}
\newtheorem{definition}{Definition}
\newtheorem{remark}{Remark}
\newtheorem*{example}{Example}

\newcommand{\rank}{\mathop{\mathrm{rank}}}

\newcommand{\R}{\mathbb{R}}
\newcommand{\N}{\mathbb{N}}
\newcommand{\Z}{\mathbb{Z}}
\newcommand{\K}{\mathcal{K}}
\newcommand{\A}{\mathcal{A}}

\begin{document}

	\title{Euler Calculus on spaces homeomorphic to definable sets and some applications}
		
	\author[E. Mac\'{\i}as-Virg\'os]{E. Mac\'{\i}as-Virg\'os}
	\address{Department of Mathematics, University of Santiago de Compostela, Spain.}
	\email{quique.macias@usc.es}

	\author[D. Mosquera-Lois]{D. Mosquera-Lois}
	\address{Institute of Mathematics, University of Santiago de Compostela, Spain.}
	\email{david.mosquera.lois@usc.es}
	
	\thanks{The first author was partially supported by MINECO-FEDER research project MTM2016-78647-P. The second author was partially supported by an IMAT research grant}

	\date{\today}
	

\begin{abstract}
We show that integration with respect to the Euler-Poincar\'e characteristic can be extended from the setting of definable sets to the setting of topological spaces homeomorphic to definable sets. We use that extension to generalize a result regarding sensor networks due to Ghrist and Baryshnikov, in order to make it more flexible in applications. Finally, we obtain both an extension and a combinatorial proof of a classical result  about the  Euler-Poincar\'e characteristic of fiber bundles.
\end{abstract}

\maketitle

\section{Introduction}

For topological spaces,  the Euler-Poincar\'e characteristic can be defined as the alternating sum of the ranks of the singular homology groups. For finite simplicial complexes this coincides with the classical combinatorial definition. As a consequence, the  Euler-Poincar\'e characteristic is a topological invariant. In the context of o-minimal structures and definable sets \cite{Viro}, a combinatorial definition of the characteristic is used \cite{Dries} in order to define an integration with respect to Euler-Poincar\'e characteristic. However, it has the drawback of no longer being a homotopy invariant. Moreover,  definable sets can be quite restrictive when studying topological questions or using topological arguments. 

In the work of Curry, Baryshnikov, Ghrist and Robinson \cite{Ghristsurvey} it is implicitly suggested that the combinatorial Euler-Poincar\'e characteristic is a topological invariant. Furthermore, spaces which are homeomorphic to definable sets are implicitly used. However, neither this is stated explicitly nor an integration theory is developed in this broader context. We aim to fill this gap by formalizing these ideas. The original formulation of Euler Calculus used subanalytic sets \cite{Schapira1991,Schapira1995}. More recent work uses definable sets. The present paper applies a theorem from Beke (if two definable sets are homeomorphic, where the homeomorphism does not need to be definable, then they have equal characteristic) to extend Euler calculus to spaces homeomorphic to definable sets. We use that extension to provide generalizations of some results regarding enumeration of objectives. Moreover, by means of that extension and the heuristics of Euler calculus, we provide a combinatorial proof of the multiplicative behavior of the Euler-Poincar\'e characteristic for fiber bundles where the total space, the base and the fiber are homeomorphic to definable sets, by discretizing Fubini's theorem and avoiding the usage of spectral sequences. \\

\noindent{\em Acknowledgement} The second author would like to thank Robert Ghrist for enlightening discussions on the topic of this work.

\section{Preliminaries}

In this section we summarize some definitions and results taken from \cite{Beke,Ghristsurvey,Dries}.

\begin{definition}
	An {\em open $n$-simplex} is the interior of an $n$-simplex $[v_0,\ldots, v_n]$, i.e., the set 
	$$\{t_0v_0+ \cdots+ t_nv_n \in \mathbb{R}^m \; \colon \; \sum_i{t_i}=1 , \; t_i> 0 \; \forall i\},$$
	where $v_0, \ldots, v_n$  are affine independent points in $\mathbb{R}^m$, $m\geq n$.
\end{definition}

\begin{definition}
	A {\em generalized simplicial complex} is a finite collection $\K $ of open simplices in $\R^n$, satisfying that given two open simplices in the complex, the intersection of their closures is the empty set or the closure of an open simplex in the complex. 
\end{definition}	

\begin{definition}
	We define the {\em Euler-Poincar\'e  characteristic} of the generalized simplicial complex $\K$ as the alternating sum of the number of open simplices in each dimension, i.e.,
	$$\chi(\K)=\sum_{i=0}^{n}(-1)^{i}c_i,$$ where 
	$c_i$ is the number of open $i$-simplices. 
\end{definition}	

\begin{definition}
	Given a topological space $X$, a {\em triangulation} is a homeomorphism $h \colon |\K | \rightarrow X$ where $|\K |$ is the geometric realization of the generalized simplicial complex $\K $.
\end{definition}

\begin{remark}
	The Euler-Poincar\'e  characteristic of an open $n$-simplex is $(-1)^n$, hence, it is not homotopy invariant. However, it is homotopy invariant on compact sets.
\end{remark}

\begin{definition}
	An {\em o-minimal structure} (over $\R$) is a colection $\A =\{\A _n\}_{n\in \N}$, where each $\A _n$ is an algebra of subsets of $\R^n$, 
	satisfying:
	\begin{enumerate}
		\item $A\in \A _n$, $B\in \A _m \Rightarrow A\times B\in \A _{n+m}$.
		\item $A\in \A _n \Rightarrow \pi(A)\in \A _{n-1}$, where $\pi$ is the projection into the first $n-1$ coordinates.
		\item $\Delta_{i,j}=\{(x_1,\ldots,x_n)\in \R^n \colon x_i=x_j\} \in \A _n$ for all $1\leq i<j\leq n$.
		\item $\{(x,y)\in \R^2\colon x<y\} \in \A _2$. 
		\item $\A _1=\{\text{finite unions of points and open intervals of $\R$}\}.$ 
	\end{enumerate}
	
	Given $A\subset \R^n$, we say that $A$ is  {\em definable in $\A $} if $A \in \A _n$. A map between definable sets is a {\em definable map} if its graph is definable. 
\end{definition}	

\begin{theorem}[{Definable triangulation theorem \cite{Dries}}] 
	Assume that $X\subset \R^n$ is a definable set and $\{X_i\}_{i=1}^{m}$ is a finite family of definable subsets of $X$. Then there exists a definable triangulation of $X$ compatible with the collection of subsets.
\end{theorem}

So, the {\em Euler-Poincar\'e characteristic of a definable set} is defined as the characteristic of the associated generalized simplicial complex.

\begin{theorem}[{\cite{Dries}}]
	Assume $X\subset \R^n$ and $Y\subset \R^m$ 
	are definable sets. Then:
	\begin{enumerate}
		\item $\chi(X \cup Y)=\chi(X)+\chi(Y)-\chi(X \cap Y).$
		\item $\chi(X\times Y)=\chi(X)\cdot \chi(Y).$
	\end{enumerate}	 
\end{theorem}

\begin{definition}
	If $X$ is a definable subset of $\R^n$, we define the $\mathbb{Z}$-module of {\em constructible functions} $CF(X)$ as follows: a function $h\colon  X \rightarrow \Z$ belongs to $CF(X)$ iff it is bounded, has compact support and the level sets are definable.
\end{definition}

\begin{theorem}[Deconstruction of constructible functions] \label{tma:deconstruction:functions} If $X$ is a definable set and $h \in CF(X)$, then $h=\sum_{\alpha}c_{\alpha}\mathbb{1}_{\sigma_{\alpha}}$, where  $c_{\alpha}\in \Z$, $\sigma_{\alpha}$ are the open simplices of a triangulation and $\mathbb{1}_{\sigma_{\alpha}}$ are the characteristic functions of $\sigma_{\alpha}$.
\end{theorem} 

\begin{definition}
	The \emph{Euler integral} is the homomorphism  
	$$\int\nolimits_X (\bullet) \, d\chi \colon CF(X) \rightarrow \Z$$ given by 
	$$\int\nolimits_X \bigg(\sum_\alpha c_\alpha\mathbb{1}_{\sigma_\alpha}\bigg) \, d\chi=\sum_\alpha c_\alpha\chi(\sigma_\alpha).$$	
\end{definition}

\begin{theorem}[{\cite{Ghristsurvey}}] The Euler integral satisfies the following properties:
	\begin{enumerate}
		\item Additivity with respect to the integrand:
		$$\int\nolimits_X\bigg(\sum_ih_i\bigg) \, d\chi =\sum_i\int\nolimits_Xh_i \, d\chi.$$
		\item Additivity with respect to the domain:
		$$\int\nolimits_{A\cup B}h \, d\chi =\int\nolimits_A h \, d\chi + \int\nolimits_B h \, d\chi - \int\nolimits_{A\cap B}h \, d\chi.$$
		\item Fubini's theorem: 
		If $f\colon  X \rightarrow Y$ is a continuous definable map and $h\in CF(X)$, then 
		$$
		\int\nolimits_X h \, d\chi=\int\nolimits_Y\bigg(\int\nolimits_{f^{-1}(y)}h \, d\chi(x)\bigg) \, d\chi(y).
		$$
	\end{enumerate}
\end{theorem}

\section{Spaces homeomorphic to definable sets and Euler Calculus}

We begin by recalling the following theorem of Beke \cite{Beke}. 

\begin{theorem}[\cite{Beke}] \label{theorem:Beke}
	Let $X$ and $Y$ be definable sets. If they are homeomorphic (not necessarily a definable homeomorphism), then $\chi(X)=\chi(Y)$. 
\end{theorem}

\begin{definition}
	Given a topological space $X$ homeomorphic to a definable set $A$, we define its {\em Euler-Poincar\'e characteristic} as $\chi(X)\coloneqq\chi(A)$. 
\end{definition}

Theorem \ref{theorem:Beke} guarantees that the Euler-Poincar\'e characteristic of   spaces homeomorphic to  definable sets is well-defined.

\begin{example}
	Examples of spaces homeomorphic to definable sets are:
	\begin{itemize}
		\item Subcomplexes of regular finite CW-complexes.
		\item Smooth compact manifolds (with or without boundary). 
		\item Topological compact manifolds (with or without boundary) of dimension less or equal to three.
	\end{itemize}
\end{example}

\begin{theorem}
	\label{tma:propiedades:conjuntos:homeo:a:definibles}
	Let $X$ and $Y$ be topological spaces   homeomorphic to definable sets. Then:
	\begin{enumerate}
		\item given a finite family $\{X_i\}_{i=1}^m$ of subsets of $X$ such that each member $X_i$ is homeomorphic to a definable set, there exists a triangulation of $X$ compatible with the collection of subsets. 
		\item If $X$ is compact, then the simplicial complex associated with the triangulation of $X$ is compact (a finite simplicial complex).
		\item If $X$ is compact, then its Euler-Poincar\'e characteristic defined counting open simplices equals its Euler-Poincar\'e characteristic using singular (co)\-ho\-mo\-lo\-gy, i.e., $$\chi(X)=\sum_n(-1)^n \rank H_n(X,\Z).$$ 
		\item Inclusion-exclusion principle: 
		\[\chi(X \cup Y)=\chi(X)+\chi(Y)-\chi(X \cap Y).\]
		\item Multiplicativity of the characteristic: $$\chi(X\times Y)=\chi(X) \cdot \chi(Y).$$
	\end{enumerate} 
\end{theorem} 

Moreover, given a space $X$ homeomorphic to a definable set $A$, the existence of the homeomorphism $\beta\colon  X \rightarrow A$ guarantees that the theorem about deconstruction of constructible functions (Theorem \ref{tma:deconstruction:functions}) extends to functions defined on $X$ such that their level sets are mapped by $\beta$ to definable subsets of $A$.

\begin{definition}
	Assume that $X$ is a  space homeomorphic to a definable set $A$ by the homeomorphism $\beta \colon  X \rightarrow A$. We define the $\Z$-module $HCF(X)$ of {\em homeo-integrable functions}, that is, functions $h\colon  X \rightarrow \Z$ which are bounded, have compact support and their level sets are homeomorphic to definable subsets of $A$ by restrictions of the homeomorphism $\beta$. 
\end{definition}


\begin{theorem} \label{tma:deconstruc_funciones:homeo:construibles}
	If $X$ is a space homeomorphic to a definable set, and $h \in HCF(X)$, then $h$ can be decomposed uniquely as a finite linear combination of characteristic functions of the open simplices of a triangulation compatible with the support of $h$, i.e.,  $h=\sum_\alpha c_\alpha\mathbb{1}_{\sigma_\alpha}$, where  $c_\alpha\in \Z$ and $\sigma_\alpha$  
	are the open simplices of the triangulation. The function $h$ can also be written as a linear combination of characteristic functions on the level sets. 
\end{theorem}

Now, we can extend Euler integration to functions in $HCF(X)$, where $X$ is a space homeomorphic to a definable set.  

\begin{definition}
	Given a space $X$  homeomorphic to a definable set, we define the {\em integral with respect to the Euler-Poincar\'e characteristic} of the function $h=\sum_\alpha c_\alpha\mathbb{1}_{\sigma_\alpha} \in HCF(X)$ as  $$\int\nolimits_X h \, d\chi\coloneqq\sum_\alpha c_\alpha\chi(\sigma_\alpha).$$ 
\end{definition}

Theorem \ref{tma:deconstruc_funciones:homeo:construibles} guarantees that this integral  is well-defined. As it happens in classical Euler Calculus, if $U$ is a subspace of $X$ homeomorphic to a definable set, then $\mathbb{1}_U\in HCF(X)$ and $\int\nolimits_X\mathbb{1}_U \, d\chi=\chi(U)$.  The results regarding the additivity of the integral also extend without major modifications. 

\begin{theorem}[Additivity of the integral with respect to the integrand]
	If $X$ is a space homeomorphic to a definable set, and $\{h_i\}_{i=1}^m$ is a finite set of functions in $HCF(X)$, then 
	$$
	\int\nolimits_X{\bigg(\sum_{i=1}^m h_i\bigg)} \, d\chi=\sum_{i=1}^m{\int\nolimits_X h_i} \, d\chi.
	$$
\end{theorem}

\begin{theorem}[Additivity of the integral with respect to the domain] \label{theorem:aditivity:domain}
	Let $X$ be a space homeomorphic to a definable set, and let $h\colon  X \rightarrow \Z$ be a function in $HCF(X)$. If $\{X_i\}_{i=1}^m$ is a finite cover of $X$ by subspaces homeomorphic to definable subsets, then 
	\begin{align*}
	\int\nolimits_{X}h \, d\chi=&\sum_i  \int\nolimits_{X_i}h \, d\chi - \sum_{i> j} \int\nolimits_{X_i\cap X_j}h \, d\chi 
	\\
	&\quad +\sum_{i>j> k} \int\nolimits_{X_i\cap X_j \cap X_k}h \, d\chi - \cdots
	+ (-1)^{m+1} \int\nolimits_{X_1 \cap  \cdots \cap X_m}h \, d\chi.
	\end{align*}
\end{theorem}

\section{Consequences}

In this section we present generalizations of Ghrist's results and two theorems regarding the Euler-Poincar\'e characteristic. First, we obtain a generalization of \cite[Theorem 3.2]{Bary2}.

\begin{theorem}
	Assume that $X$ is a space homeomorphic to a definable set and $h\colon  X \rightarrow \N$ is a function in $HCF(X)$, $h=\sum_{\alpha\in I}\mathbb{1}_{U_\alpha}$, where the supports $\{U_\alpha\}_{\alpha \in I}$ are subspaces homeomorphic to definable subsets. If the Euler-Poincar\'e characteristic of all the supports  $U_\alpha$ is $N$ and different from zero, then the number of observables is
	$$\vert I\vert=\frac{1}{N} \int\nolimits_X h \, d\chi.$$
\end{theorem} 

\begin{remark}
	Another results about sensor networks presented in \cite{Bary2} can be generalized to spaces homeomorphic to definable sets by discretizing time.
\end{remark}

Using the generalization of Euler Calculus to spaces homeomorphic to definable sets we prove now the relationship between the Euler-Poincar\'e characteristics of the spaces involved in a locally trivial fiber bundle. The proof is mainly of combinatorial nature, avoiding the use of spectral sequences as in \cite[Theorem 1, p. 481]{Spanier} or \cite{Serre}, and improves \cite{Spanierjoven2}. Moreover, the theorem below generalizes related results since it works for all spaces homeomorphic to definable sets and we only need the compactness assumption on the base space.

\begin{theorem}
	Let $p \colon  E \rightarrow B$ be a locally trivial fiber bundle where 
	\begin{enumerate}
		\item[(a)] the total space $E$, the base $B$ and the fiber $F$ are homeomorphic to definable sets;
		\item[(b)] $B$ is compact.
	\end{enumerate}
	Then $\chi(E)=\chi(B)\cdot \chi(F)$.
\end{theorem}

\begin{proof}
	It follows from the Lebesgue number lemma, Theorem \ref{theorem:aditivity:domain} and the multiplicativity of the characteristic. Being more precise, since $B$ is compact we can cover it with a finite trivialization by open sets $\{U_j\}_{j\in J}$. Consider a triangulation of $B$ and subdivide it until we have a finite refinement $\{B_i\}_{i\in I}$ by subspaces  homeomorphic to definable subsets. Assume $I=\{1,\ldots,m\}$. Then we have:
	\begin{align*}
	\chi(E)= &\int\nolimits_X\mathbb{1}_E \, d\chi\\
	= &\sum_i \int\nolimits_{p^{-1}(B_i)}\mathbb{1}_E \, d\chi - \sum_{i\neq j} \int\nolimits_{p^{-1}(B_i)\cap p^{-1}(B_j)}\mathbb{1}_E \, d\chi 
	\\
	&\quad+\sum_{i> j > k} \int\nolimits_{p^{-1}(B_i)\cap p^{-1}(B_j) \cap p^{-1}(B_k)}\mathbb{1}_E \, d\chi \\
	&\quad - \cdots + (-1)^{m+1} \int\nolimits_{p^{-1}(B_1) \cap  \cdots \cap p^{-1}(B_m)}\mathbb{1}_E \, d\chi\\
	& \text{}\\
	= &\sum_i \int\nolimits_{p^{-1}(B_i)}\mathbb{1}_E \, d\chi - \sum_{i> j} \int\nolimits_{p^{-1}(B_i\cap B_j)}\mathbb{1}_E \, d\chi \\
	&\quad+\sum_{i> j > k} \int\nolimits_{p^{-1}(B_i\cap B_j \cap B_k)}\mathbb{1}_E  \, d\chi \\
	&\quad - \cdots
	+ (-1)^{m+1} \int\nolimits_{p^{-1}(B_1 \cap  \cdots \cap B_m)}\mathbb{1}_E  \, d\chi\\
	& \text{}\\
	= &\sum_i  \chi(p^{-1}(B_i)) - \sum_{i> j} \chi(p^{-1}(B_i\cap B_j))\\
	&\quad+\sum_{i> j > k} \chi(p^{-1}(B_i\cap B_j \cap B_k))\\
	&\quad - \cdots + (-1)^{m+1} \chi(p^{-1}(B_1 \cap  \cdots \cap B_m))\\
	& \text{}\\
	= &\sum_i  \chi(B_i \times F) - \sum_{i> j} \chi((B_i\cap B_j)\times F)\\
	&\quad+\sum_{i> j > k } \chi((B_i\cap B_j \cap B_k) \times F )\\
	&\quad - \cdots + (-1)^{m+1} \chi((B_1 \cap  \cdots \cap B_m)\times F)\\
	& \text{}\\
	= &\sum_i  \chi(B_i) \cdot \chi(F) - \sum_{i> j} \chi(B_i\cap B_j) \cdot \chi(F) \\
	&\quad +\sum_{i> j > k} \chi(B_i\cap B_j \cap B_k)\cdot \chi(F)\\
	&\quad  - \cdots + (-1)^{m+1} \chi(B_1 \cap  \cdots \cap B_m) \cdot \chi(F)\\
	& \text{}\\
	= &\chi(F) \cdot \bigg( \sum_i  \chi(B_i) - \sum_{i> j} \chi(B_i\cap B_j) \\
	&\quad+\sum_{i> j > k} \chi(B_i\cap B_j \cap B_k)\\
	&\quad - \cdots + (-1)^{m+1} \chi(B_1 \cap  \cdots \cap B_m) \bigg)\\
	& \text{}\\
	= & \chi(F) \cdot \chi(B). 
	\end{align*}	
\end{proof}

Notice that this result can be thought as being essentially Fubini's theorem, where the characteristic $\chi(F)$ is a constant.

By using the same ideas, a purely combinatorial proof of the Riemann-Hurwitz Formula, alternative to the one presented in \cite{Viro},  can also be given. 

\begin{theorem}[Riemann-Hurwitz]
	If $f\colon  \widetilde{S} \rightarrow S$ is a ramified covering between two compact Riemann surfaces without boundary, i.e., a non constant holomorphic map, then
	\begin{equation*}
	\chi(\widetilde{S})=n\chi(S) -\sum_{x\in S}(e_f(x)-1),
	\end{equation*}
	where $e_f(x)$ is the ramification index.
\end{theorem}




\bibliographystyle{plain}
\bibliography{biblio_EC}

\end{document}